\documentclass[11pt,oneside,reqno]{amsart}
\usepackage{mathpazo} 
\normalfont
\usepackage[T1]{fontenc}

\usepackage[utf8]{inputenc}
\usepackage{amsmath}
\usepackage{amsfonts}
\usepackage{amssymb}
\usepackage{amsthm}
\usepackage{tikz}
\usepackage{tikz-cd}
\usepackage[all]{xy}
\usepackage[margin=1.2in]{geometry}
\usepackage{amscd}
\usepackage[shortlabels]{enumitem}

\DeclareMathOperator{\Hom}{Hom}

\newcommand{\angles}[1]{\left\langle #1 \right\rangle}

\input xy
\xyoption{all}
\theoremstyle{definition}
\newtheorem{mydef}{\textbf{Definition}}[section]
\newtheorem{myeg}[mydef]{\textbf{Example}}

\newtheorem{rmk}[mydef]{\textbf{Remark}}

\theoremstyle{plain}
\newtheorem{mythm}[mydef]{\textbf{Theorem}}

\newtheorem*{nothma}{\textbf{Theorem A}}
\newtheorem*{nothmb}{\textbf{Theorem B}}
\newtheorem*{nothmc}{\textbf{Theorem C}}

\newtheorem{lem}[mydef]{\textbf{Lemma}}
\newtheorem{pro}[mydef]{\textbf{Proposition}}

\newtheorem{cor}[mydef]{\textbf{Corollary}}

\newcommand{\T}{\mathbb{T}}

\newcommand{\multipleaffil}[3]{%
  \address{%
    \begin{minipage}[t]{\textwidth}
      #1 \\
      #2 \\
      #3 
    \vspace{0.1cm}
    \end{minipage}
  }
}

\begin{document}

\title{Matroidal representations of low rank}
\author{Jaiung Jun}
\multipleaffil{Department of Mathematics, State University of New York at New Paltz, NY, USA}{and}{Institute for Advanced Study, Princeton, NJ, USA}
\email{junj@newpaltz.edu, junj@ias.edu}

\author{Kalina Mincheva}
\address{Department of Mathematics, Tulane University, New Orleans, LA 70118, USA}
\email{kmincheva@tulane.edu}

\author{Jeffrey Tolliver}
\address{}
\email{jeff.tolli@gmail.com}
\makeatletter
\@namedef{subjclassname@2020}{%
	\textup{2020} Mathematics Subject Classification}
\makeatother

\thanks{}

\begin{abstract}
We study tropical subrepresentations of the Boolean regular representation $\mathbb{B}[G]$ of a finite group $G$. These are equivalent to the matroids on ground set $G$ for which left-multiplication by each element of $G$ is a matroid automorphism. We completely classify the tropical subrepresentations of $\mathbb{B}[G]$ for rank 3. When $G$ is an abelian group, our approach can be seen as a generalization of Golomb rulers. In doing so, we also introduce an interesting class of matroids obtained from equivalence relations on finite sets. 
\end{abstract}

\maketitle

\tableofcontents


\section{Introduction}

Tropical geometry is a recent field of algebraic geometry with the aim to study algebro-geometric structures through associated combinatorial structures. 
Classically, tropical geometry has focused on the geometry and has not taken full advantage of the underlying semiring algebra until now. Since (valuated) matroids can be identified with tropical linear spaces, it is natural to ask whether the theory of semirings provides an algebraic foundation for matroids. More precisely, one may ask whether or not one can see matroids (and their variations) as certain modules over semirings. 


There have been various algebraic approaches to matroid theory in recent years.
Our main sources of motivation are Frenk's thesis \cite{frenk2013tropical} and works inspired by it, such as \cite{giansiracusa2018grassmann} and \cite{crowley2020module}, where the authors view (valuated) matroids as modules of Grassmann algebras over idempotent semifields. Recent work of Baker and Bowler \cite{baker2019matroids} provides a unifying framework for various generalizations of matroids. One may view a module-theoretic approach to matroids as contributing further to the generality of the Baker-Bowler theory. A different categorical approach developed by Nakamura and Reyes \cite{nakamura2023categories} and Beardsley and Nakamura \cite{beardsley2024projective} views (simple) matroids as $\mathbb{F}_1$-modules via their connection to projective geometries.



In \cite{giansiracusa2020matroidal}, Giansiracusa and Manaker introduce a notion of tropical representations. For a group $G$ and an idempotent semifield $K$, and a linear representation $\rho: G \to \text{GL}_n(K)$, a tropical representation is a $G$-invariant tropical linear space $T$ in $K^n$. When $K=\mathbb{B}$, the Boolean semifield, this is equivalent to a group homomorphism from $G$ to $\text{Aut}(M_T)$, where $M_T$ is the underlying matroid of $T$ and $\text{Aut}(M_T)$ is the automorphism group of $M_T$ - hence the name, a matroidal representation. Based on the fact that any $\text{GL}_n(K)$ for an idempotent semifield $K$ is a semidirect product of $S_n$ and $(K^\times)^n$, Giansiracusa and Manaker define also a notion of tropicalization for monomial representations, and they study realizability of tropical representations. An equivalent description of dimension $d$ tropical subrepresentations of a linear representation is the fixed points of the $G$-action on the Dressian $\text{Dr}(d,n)$ (see, \cite[Corollrary 3.0.4]{giansiracusa2020matroidal}).

Giansiracusa and Manaker show that there is very interesting interplay between group structures and matroid structures. For instance, if a finite group $G$ is abelian, $G$ is cyclic if and only if the tropical representation $U_{2,G}$ is realizable. Based on computing small examples, they conjecture that there is only one two-dimensional subrepresentation of the regular representation $\mathbb{B}[\mathbb{Z}_n]$ if and only if $n$ is prime. In \cite{marcus2024tropical}, Marcus and Phillips prove the conjecture. In this paper we extend some of their work further.






\subsection{Summary of results}

Our main result is to give a classification of group actions on matroids of rank 2 and 3. \footnote{See Section \ref{subsection: matroidal rep} for the precise definition of matroidal representation.}

\begin{nothma}[Proposition~\ref{prop:rk2matroids}, Corollary~\ref{corollary: thmc}]
Let $G$ be a finite group.
\begin{enumerate}
\item
Let $X$ be a $G$-set.  There is a one-to-one correspondence between $G$-invariant loopless rank $2$ matroid structures on $X$ and $G$-invariant nontrivial equivalence relations on $X$.
\item
There is a one-to-one correspondence between rank $3$ $G$-invariant matroids whose ground $G$-set is $G$ and pairs consisting of a proper subgroup $H$ and a nontrivial equivalence relation $\sim$ on $G / H - \bar{1}$ such that the following hold
\begin{enumerate}
    \item 
    If $\bar{a} \sim \bar{b}$ but $\bar{a}\neq \bar{b}$ then $\overline{a^{-1}} \sim \overline{a^{-1}b}$.
    \item 
    If $\bar{a}\sim \bar{b}$ and $h\in H$ then $\overline{ha} \sim \overline{hb}$.
\end{enumerate}
Furthermore, $H$ is the trivial group if and only if the matroid is simple.
\end{enumerate}
\end{nothma}

The above theorem generalizes \cite[Theorems  3.5 and 4.5]{marcus2024tropical} although our approach is completely different. See Remark \ref{rmk: implying MP24}.


We then focus on some special cases such as $G=\mathbb{Z}_p$. In this case, $G$-invariant matroid structures on $G$ may have an interesting connection with number theory (see Proposition \ref{proposition: cyclic}). To study the case when $G$ is an abelian group, we introduce a notion of \emph{distinct difference  system} (Definition \ref{definition: difference system}) of an abelian group $G$ - this is a collection of subsets $\{S_i\}$ of $G$ satisfying some conditions. Note that distinct difference systems can be seen as a generalization of \emph{modular Golomb ruler} (See Corollary \ref{cor: distinct difference property of equivalence class} and discussion after).

\begin{nothmb}[Theorems \ref{theorem: abelian case translation} and \ref{theorem: Classification in the prime cyclic case}]
Let $k\geq 3$.  Let $G$ be a finite abelian group with at least $k$ elements.  
\begin{enumerate}
    \item 
Fix a distinct difference system $S_1,\ldots S_l$.  Define a subset of $G$ to be a basis if it has $k$ elements and is not a translate of a subset of some $S_i$.  Then this yields a simple rank $k$ $G$-invariant matroid structure on $G$. 
    \item 
If $G$ is a prime cyclic group and $k = 3$ then all simple rank $3$ $G$-invariant matroid structures on $G$ arise in this way.    
\end{enumerate}
\end{nothmb}




Finally, we link certain equivalence relations on a finite set $E$ and matroid structures on $E$. We also provide $G$-invariant version of this correspondence.



\begin{nothmc}[Proposition \ref{prop: correspondence between loopless and simple matroids}]
Let $E$ be a finite set.  Let $k\geq 2$.  
\begin{enumerate}
    \item 
There is a one-to-one correspondence between loopless rank $k$ matroid structures on $E$ and pairs consisting of a nontrivial equivalence relation $\sim$ on $E$ and a simple rank $k$ matroid structures on $E/\sim$.
    \item 
Let $G$ be a group and fix a $G$-action on $E$.  There is a one-to-one correspondence between $G$-invariant loopless rank $k$ matroid structures on $E$ and pairs consisting of a nontrivial $G$-invariant equivalence relation $\sim$ on $E$ and a simple rank $k$ $G$-invariant matroid structures on $E/\sim$.
\end{enumerate}   
\end{nothmc}






\bigskip

\textbf{Acknowledgment} J.J. acknowledges AMS-Simons Research Enhancement Grant for Primarily Undergraduate Institution (PUI) Faculty during the writing of this paper, and parts of this research was done during his visit to the Institute for Advanced Study supported by the Bell System Fellowship Fund. K.M. acknowledges the support of the Simons Foundation, Travel Support for Mathematicians.

\section{Preliminaries}\label{section: preliminaries}


A \textit{semiring} is a set $R$ with two binary operations (addition $+$ and multiplication $\cdot$ ) satisfying the same axioms as rings, except the existence of additive inverses. In this paper, a semiring is always assumed to be commutative. A semiring $(R,+,\cdot)$ is \emph{semifield} if $(R\backslash\{0_R\},\cdot)$ is a group. A semiring $R$ is said to be \emph{zero-sum free} if $a+b=0$ implies $a=b=0$ for all $a,b \in R$. 

We will denote by $\mathbb{B}$ the semifield with two elements $\{1,0\}$, where $1$ is the multiplicative identity, $0$ is the additive identity and $1+1 = 1$. 
The {\it tropical semifield}, denoted $\mathbb{T}$, is the set $\mathbb{R}  \cup \{-\infty\} $ with the $+$ operation to be the maximum and the $\cdot$ operation to be the usual addition, with $-\infty = 0_\mathbb{T}$. 

\begin{mydef}\label{definition: bend relation}
	Let $f$ be a polynomial in the polynomial semiring $\T[x_1,\dots, x_n]$. The \emph{bend relations} of $f$ is the set of equivalences $\{ f\sim f_{\hat{i}}\}$, where $f_{\hat{i}}$ is the polynomial $f$ after removing its $i$-th monomial. For a subset $S$  of $\T[x_1,\dots, x_n]$ the \emph{bend congruence} of $S$ is the congruence generated by the bend relations of all $f \in S$.
\end{mydef}

\begin{mydef}
By a \emph{tropical linear space} $V$ in $\mathbb{T}^n$, we mean the subset of $\mathbb{T}^n$ obtained in the following way: there exists a congruence relation on $\mathbb{T}[x_1,\dots,x_n]$ generated by $\{L_i \sim L_j\}_{i,j \in I}$, where $L_i \in \mathbb{T}[x_1,\dots,x_n]$ are linear polynomials, such that 
\[
V=\Hom_\mathbb{T}(\mathbb{T}[x_1,\dots,x_n]/\sim,\mathbb{T}). 
\]
and the congruence relation $\sim$ is a bend relation as in Definition \ref{definition: bend relation} which comes from circuits of a matroid.
\end{mydef}

For an equivalent definition of tropical linear spaces we point the reader to Section 2.1 in \cite{giansiracusa2020matroidal}.

\subsection{Matroidal representation of finite groups} \label{subsection: matroidal rep}

\begin{mydef}\label{def:TR}
Let $K$ be an idempotent semifield. For a given (linear) representation $\phi:G \to \text{GL}_n(K)$, a \emph{tropical representation} of $G$ is a $G$-invariant tropical linear space $V$ in $K^n$. 
By the rank (or dimension) of a tropical representation we mean the rank of the tropical linear space, or equivalently, the rank of the corresponding matroid.
\end{mydef}

In \cite{giansiracusa2020matroidal} Giansiracusa and Manaker use the term representation to mean both the linear representation $\phi$ of $G$ and the tropical representation from Defintion~\ref{def:TR}.

In the case where $K = \mathbb{B}$, and $M_V$ is the matroid associated to a tropical representation $V$ in $\mathbb{B}^n$, Giansiracusa and Manaker proved that the tropical representation is equivalent to a group homomorphism $G \to \text{Aut}(M_V)$, where $\text{Aut}(M_V)$ is the automorphism group of the matroid $M_V$. As such, Giansiracusa and Manaker referred to tropical representations as matroidal representations when $K=\mathbb{B}$.

\section{Group actions on matroids of low rank}
Tropical representations of rank $2$ or $3$ over $\mathbb{B}$ have recently been studied by \cite{marcus2024tropical}.  Such representations correspond to group actions on matroids of low rank.  In this section, we prove further results on such group actions. When the context is clear, we often write $A+x$ (resp.~$A-x$) for $A\cup \{x\}$ (resp.$A - \{x\}$, especially when we work on independent subsets/ bases/circuits of matroids. When a matroid is $G$-invariant, for an abelian group $G$, to avoid any confusion, we use $\cup$ notation instead of $+$.

\subsection{Group actions on matroids of rank 2}

\cite[Theorem 1]{marcus2024tropical} relates subgroups of $G$ to $G$-invariant rank 2 matroids with base set $G$.  By working with the matroid language, we can simplify the argument and extend it to all $G$-actions on rank 2 matroids without loops.

The key idea is a relationship between rank $2$ matroids and equivalence relations.  For motivation (in the linear matroid case), observe that two nonzero vectors are linearly dependent if and only if they determine the same point in projective space, and this is an equivalence relation.

\begin{lem}\label{lemma: equivalence relation and rank 2}
Let $M$ be a matroid without loops on $E$.  Define a relation $\sim$ on $E$ by $x\sim y$ if $x = y$ or $\{x, y\}$ is dependent.  Then $\sim$ is an equivalence relation.
\end{lem}
\begin{proof}
Reflexivity and symmetry are clear.  These properties imply the transitive law in the case two of the three elements involved are equal. So, suppose $x,y,z\in M$ are distinct with $x\sim y$ and $y\sim z$.  By the loopless assumption, $\{x, y\}$ and $\{y, z\}$ are circuits.  By circuit elimination, $\{x, z\}$ is dependent so $x\sim z$.
\end{proof}


In what follows, by a \emph{nontrivial equivalence relation} we mean an equivalence relation with at least two equivalence classes.

\begin{lem}\label{lemma: coorespondence}
Let $E$ be a nonempty finite set. The equivalence relations defined in Lemma \ref{lemma: equivalence relation and rank 2} induce a one-to-one correspondence between loopless rank $2$ matroids with ground set $E$ and nontrivial equivalence relations on $E$. 
\end{lem}
\begin{proof}
We consider the map from loopless rank $2$ matroids with ground set $E$ to equivalence relations on $E$, which sends a matroid to the equivalence relation $\sim$, where $x\sim y$ means $x=y$ or $\{x, y\}$ is dependent.  Observe that an equivalence relation constructed this way is nontrivial - by the rank $2$ assumption, there is some independent set $\{x, y\}$ so $x\not\sim y$.

We first show injectivity.  Observe that a matroid structure is determined by the set of bases, which is the set of independent sets of size $2$.  Equivalently, it is determined by specifying which sets of size $2$ are dependent, and this is what the equivalence relation $\sim$ does.

For surjectivity, let $\sim$ be a nontrivial equivalence relation.  Call a subset $I\subseteq E$ independent if it is contained in a set of the form $\{x, y\}$ for some $x\not\sim y$.  Clearly a subset of an independent set is independent. Since $\sim$ is nontrivial, for all $z\in E$, there exists $w\in E$ such that $z\not\sim w$, i.e. $\{z, w\}$ is independent.  This has two consequences - any singleton $\{z\}$ is independent, and any independent set of size less than $2$ can be extended to one of size $2$. 

Let $\mathcal{I}$ be the set of independent subsets defined above from $\sim$. We claim that $\mathcal{I}$ defines a matroid on $E$. In fact, we only have to check the exchange axiom. Let $I, J$ be independent sets with $|I| < |J|$.  The empty set case is easy, so assume $|I| = 1$ and $|J| = 2$.  Let $x, y, z$ be such that $J = \{x,y\}$ and $I=\{z\}$.  Then $x\not\sim y$, and hence $z\not\sim x$ or $z\not\sim y$ since otherwise $x \sim y$ by the transitivity of $\sim$. It follows that either $\{x, z\}$ or $\{y, z\}$ is independent, showing the exchange axiom. Moreover, by applying the transitive law together with the assumption that an equivalence relation is nontrivial, one sees that $M$ is loopless. Therefore, $M=(E,\mathcal{I})$ is a rank $2$ matroid without loops. 

Finally, it is clear from construction that this matroid structure maps to $\sim$, establishing surjectivity.
\end{proof}

\begin{mydef}
Let $G$ be a finite group and $X$ be a finite $G$-set. Let $M$ be a matroid on $X$. We say that $M$ is $G$-invariant if for any $g \in G$ and a base $B \in \mathcal{B}$, we have $gB \in \mathcal{B}$. 
\end{mydef}

\begin{pro}\label{prop:rk2matroids}
Let $G$ be a finite group and $X$ be a finite $G$-set.  There is a one-to-one correspondence between $G$-invariant loopless rank $2$ matroid structures on $X$ and $G$-invariant nontrivial equivalence relations on $X$.
\end{pro}
\begin{proof}
From Lemma \ref{lemma: coorespondence}, we have a one-to-one correspondence between loopless rank $2$ matroid structures on $X$ and nontrivial equivalence relations on $X$.  We must show that the matroid structure is $G$-invariant if and only if the equivalence relation is $G$-invariant.  Equivalently, we must show that the following two statements are equivalent for $g\in G$ and $x,y\in X$:
\begin{enumerate}
    \item If $x\sim y$ then $gx\sim gy$.
    \item If $\{x, y\}$ is a basis then $\{gx, gy\}$ is a basis.
\end{enumerate}
Observe that the first statement is equivalent (via use of $g^{-1}$) to saying that $x\not\sim y$ implies $gx\not\sim gy$.  The equivalence of the two statements follows by observing that $\{x,y\}$ is a basis if and only if $x\not\sim y$.
\end{proof}

\begin{rmk}\label{rmk: implying MP24}
Observe that $G$-invariant nontrivial equivalence relations on $G$ are precisely the congruence kernel \footnote{ The congruence-kernel of $\phi: R \to R'$  is the pullback of the trivial congruence on $R'$.} of maps $G\rightarrow G/H$ with $H$ a proper subgroup of $G$ and $G/H$ is the set of cosets.  Observe also that any $G$-invariant rank $2$ matroid with ground $G$-set $G$ is loopless, since if there were a loop, by transitivity of the action all elements would be loops.  So Proposition \ref{prop:rk2matroids} implies that $G$-invariant rank $2$ matroids with ground $G$-set equal to $G$ are in one-to-one correspondence with proper subgroups of $G$.  This is essentially \cite[Theorem 3.5]{marcus2024tropical}.
\end{rmk}

\subsection{Group actions on matroids of rank 3}

We next turn to matroids of rank $3$.  To simplify the problem, we will later focus on the case where the matroid is simple and the ground set of the matroid is a cyclic group. Recall that a simple matroid is a matroid such that all circuits contain at least three elements.

In this subsection, we let $G$ be a finite group, $H$ a proper subgroup, and $G/H$ the set of cosets. We will denote $aH \in G/H$ by $\bar{a}$. Finally, we will assume that $[G:H] \geq 3$ since we will consider rank $3$ matroid structures on $G/H$. 

\begin{lem}\label{lemma: rank 3 necessary}
Consider a simple rank $3$ $G$-invariant matroid $M$ whose ground $G$-set is $G / H$.  We define a relation on $G / H - \{\bar{1}\}$ by $\bar{a}\sim \bar{b}$ if $\bar{a} = \bar{b}$ or if $\{\bar{1}, \bar{a}, \bar{b}\}$ is dependent.  Then
\begin{enumerate}
    \item 
$\sim$ is a nontrivial equivalence relation.
    \item 
If $\bar{a} \sim \bar{b}$ but $\bar{a}\neq \bar{b}$ then $\overline{a^{-1}} \sim \overline{a^{-1}b}$.
    \item 
For any $h\in H$ and $a,b\in G - H$, if $\bar{a}\sim\bar{b}$ then $\overline{ha} \sim \overline{hb}$.
    \item 
The bases of the matroid $M$ are precisely the sets of the form $\{\bar{a}, \bar{b}, \bar{c}\}\subseteq G/H$, where all three elements are distinct and $\overline{a^{-1}b} \not\sim \overline{a^{-1}c}$.
\end{enumerate}
\end{lem}
\begin{proof}
(1): The relation $\sim$ is clearly reflexive and symmetric and the transitive law holds if two of the elements involved are equal.  Let $\bar{a}\sim \bar{b}$ and $\bar{b}\sim \bar{c}$ with $\bar{a}, \bar{b}, \bar{c}$ distinct elements not equal to $\bar{1}$.  Then $\{\bar{1}, \bar{a}, \bar{b}\}$ and $\{\bar{1}, \bar{b}, \bar{c}\}$ are dependent, and by the simplicity assumption, they are in fact circuits.  By circuit elimination, $\{\bar{1}, \bar{a}, \bar{c}\}$ is dependent so $\bar{a}\sim \bar{c}$.  This establishes transitivity.  We will show $\sim$ is nontrivial later after the proof of $(4)$.

(2): If $\bar{a}\sim \bar{b}$ and $\bar{a} \neq \bar{b}$, then $\{\bar{1}, \bar{a}, \bar{b}\}$ is dependent.  By $G$-invariance, $\{\overline{a^{-1}}, \bar{1}, \overline{a^{-1}b}\}$ is also dependent so $\overline{a^{-1}} \sim \overline{a^{-1}b}$ since $\bar{a}\neq \bar{b}$ so $\overline{a^{-1}b}\neq 1$.

(3): If $\bar{a}\sim\bar{b}$ but $\bar{a}\neq\bar{b}$, then $\{\bar{1}, \bar{a}, \bar{b}\}$ is dependent, and acting on it by $h$ tells us $\{\bar{1}, \overline{ha}, \overline{hb}\}$ is dependent so $\overline{ha} \sim \overline{hb}$.  On the other hand, if $\bar{a} = \bar{b}$, acting on both sides by $h$ (and using that $=$ implies $\sim$) yields $\overline{ha} \sim \overline{hb}$

(4): By $G$-invariance $\{\bar{a}, \bar{b}, \bar{c}\}$ is a basis if and only if $\{\bar{1}, \overline{a^{-1}b}, \overline{a^{-1}c}\}$ is.  This is equivalent to $\{\bar{1}, \overline{a^{-1}b}, \overline{a^{-1}c}\}$ not being dependent, which is equivalent to $\overline{a^{-1}b} \not\sim \overline{a^{-1}c}$. Now, if $\sim$ were trivial, then the last claim would imply there are no bases, a contradiction.
\end{proof}

\begin{rmk}
The second part of the above lemma is essentially the same as \cite[Proposition 4.1]{marcus2024tropical}, but the interpretation in terms of equivalence relations is new.
\end{rmk}


\begin{lem}\label{lemma: implying}
Let $G$ be a group and $H$ be a subgroup.  Let $\sim$ be a nontrivial equivalence relation on $G/H - \{\bar{1}\}$ such that if $\bar{a} \sim \bar{b}$ but $\bar{a}\neq \bar{b}$ then $\overline{a^{-1}} \sim \overline{a^{-1}b}$.  Then, that $\overline{g^{-1}a}\sim\overline{g^{-1}b}$ for all $g\not\in aH \cup bH$ implies $\bar{a}=\bar{b}$.
\end{lem}
\begin{proof}
Suppose $\overline{g^{-1}a}\sim\overline{g^{-1}b}$ for all $g\not\in aH \cup bH$ and $\bar{a} \neq \bar{b}$.  For any such $g$, we have $\overline{g^{-1}a}\neq\overline{g^{-1}b}$, and hence 
\begin{equation}\label{eq: eq9}
\overline{b^{-1}a} \sim \overline{b^{-1}g}
\end{equation}
for all $g\not\in aH \cup bH$. For $g \in aH$, we have $\overline{b^{-1}a} = \overline{b^{-1}g}$, and hence by reflexivity, \eqref{eq: eq9} also holds for $g\in aH$. Therefore, \eqref{eq: eq9} holds for all $g$ such that $\overline{b^{-1}g} \neq \bar{1}$. But this implies all elements are equivalent to $\overline{b^{-1}a}$, contradicting nontriviality of $\sim$. 
\end{proof}

\begin{pro}\label{proposition: simple rank 3 classification}
There is a one-to-one correspondence between simple rank $3$ $G$-invariant matroids whose ground $G$-set is $G / H$ and nontrivial equivalence relations $\sim$ on $G / H - \{\bar{1}\}$ satisfying the following two conditions.
\begin{enumerate}
    \item 
    If $\bar{a} \sim \bar{b}$ but $\bar{a}\neq \bar{b}$ then $\overline{a^{-1}} \sim \overline{a^{-1}b}$.
    \item 
    If $\bar{a}\sim \bar{b}$ and $h\in H$ then $\overline{ha} \sim \overline{hb}$.
\end{enumerate}
\end{pro}
\begin{proof}
We have seen in Lemma \ref{lemma: rank 3 necessary} that given such a matroid, we can construct an equivalence relation with the desired properties.  This assignment is injective 
in view of Lemma~\ref{lemma: rank 3 necessary} (4). So we only need to prove surjectivity.

Let $\sim$ be a nontrivial equivalence relation on $G / H - \{\bar{1}\}$ satisfying the two conditions in the statement of the proposition. 
We define a basis to be a 3-element subset $\{\bar{a}, \bar{b}, \bar{c}\}\subseteq G/H$ such that $\overline{a^{-1}b} \not\sim \overline{a^{-1}c}$. 

We need to show that the 3-element sets satisfying the above property are indeed bases of a matroid. To show the notion of a basis is well-defined, there are two potential issues we need to rule out. First, we show that whether the equation $\overline{a^{-1}b} \sim \overline{a^{-1}c}$ holds does not depend on the choice of representatives. 
For this, we only need to show it does not depend on the choice of $a$, so let $a, a'$ be such that $\bar{a} = \bar{a'}$, that is, there exists $h\in H$ with $a' = ah$.  Then $a'^{-1}b = h^{-1}a^{-1}b$ and similarly, $a'^{-1}c = h^{-1}a^{-1}c$.  By condition (2) we have that $\overline{a^{-1}b} \sim \overline{a^{-1}c}$ implies $\overline{ha^{-1}b} \sim \overline{ha^{-1}c}$, or equivalently $\overline{a'^{-1}b} \sim \overline{a'^{-1}c}$. 

Secondly, we need to show that if a 3-element sets $\{\bar{a}, \bar{b}, \bar{c}\}\subseteq G/H$ satisfies the property $\overline{a^{-1}b} \not\sim \overline{a^{-1}c}$, it is still true after reordering the elements.
It is clear that whether $\overline{a^{-1}b} \sim \overline{a^{-1}c}$ holds is invariant under swapping $\bar{b}$ and $\bar{c}$.  Since two transpositions generate the permutation group $S_3$, we must show it is invariant under swapping $\bar{a}$ and $\bar{b}$.  That is we must show that $\overline{a^{-1}b} \sim \overline{a^{-1}c}$ implies $\overline{b^{-1}a} \sim \overline{b^{-1}c}$.  Write $x = a^{-1}b$ and $y = a^{-1}c$.  Then $x^{-1} = b^{-1}a$ and $x^{-1}y = b^{-1}c$.  Since $\bar{b} \neq \bar{c}$, we have $\bar{x} \neq \bar{y}$. Therefore, $\bar{x} \sim \bar{y}$ implies $\overline{x^{-1}} \sim \overline{x^{-1}y}$ by the property (1), and we obtain that the notion of basis is well-defined in the sense that checking the condition is independent of ordering of the elements in a basis. 

Next we claim that for any two distinct $\bar{a}, \bar{b} \in G/H - \{\bar{1}\}$, $\{\bar{a}, \bar{b}\}$ is contained in a basis.  By Lemma \ref{lemma: implying}, there is some $g$ such that $\overline{g^{-1}a}\not\sim \overline{g^{-1}b}$. Thus $\{\bar{g},\bar{a},\bar{b}\}$ is a basis, and hence this establishes the claim.  Furthermore, this implies that the matroid is simple (provided it is actually a matroid). 

We call a subset dependent if it is not contained in a basis, and define a circuit to be a minimal dependent set. We will show these are circuits of a matroid. First, by the claim in the previous paragraph, all circuits have size at least 3.  Clearly no circuit is contained in another circuit.  We must show that if $C_1$ and $C_2$ are distinct circuits with a common element $\bar{x}$, then $C_1 \cup C_2 - \{\bar{x}\}$ is dependent.

By construction all subsets of size at least $4$ are dependent, so the result holds if 
\[
|C_1 \cup C_2| - 1\geq 4.
\]
This is always true if $|C_1| \geq 4$ or $|C_2| \geq 4$.  Since any two circuits have at least $3$ elements, it remains to show that they have exactly $3$ elements. 
Note that $|C_1 \cup C_2| = 6 - |C_1 \cap C_2|$ so the result holds when $|C_1 \cap C_2| \leq 1$.  Since both sets are distinct, the remaining case is that they overlap in two elements.  So we may suppose $C_1 = \{\bar{x}, \bar{a}, \bar{b}\}$ and $C_2 = \{\bar{x}, \bar{b}, \bar{c}\}$.

Since $\{\bar{x}, \bar{a}, \bar{b}\}$ is not a basis, we must have that $\overline{b^{-1}a} \sim \overline{b^{-1}x}$.  Similarly, $\overline{b^{-1}c} \sim \overline{b^{-1}x}$, so $\overline{b^{-1}a} \sim \overline{b^{-1}c}$.  Thus $\{\bar{b}, \bar{a}, \bar{c}\} = C_1 \cup C_2 - \{\bar{x}\}$ is not a basis. Similarly, $\{\bar{x},\bar{a},\bar{c}\} = C_1\cup C_2 - \{\bar{b}\}$ is not a basis, which implies the circuit elimination property.

We have seen the matroid is simple and has rank $3$, so it remains to check $G$-invariance.  Suppose $\{\bar{a}, \bar{b}, \bar{c}\}$ is not a basis.  Then $\overline{a^{-1}b} \sim \overline{a^{-1}c}$.  For any $g\in G$, we have $\overline{(ga)^{-1}gb} \sim \overline{(ga)^{-1}gc}$.  Thus $\{\bar{ga}, \bar{gb}, \bar{gc}\}$ is not a basis.
\end{proof}

Of course in the case that the ground $G$-set is $G$, the last axiom in Proposition \ref{proposition: simple rank 3 classification} we require for $\sim$ is vacuous.  

\begin{rmk}\label{remark: equivalence class of inverse}
Property (1) in Proposition \ref{proposition: simple rank 3 classification} stating that if $\bar{a} \sim \bar{b}$ but $\bar{a}\neq \bar{b}$ then $\overline{a^{-1}} \sim \overline{a^{-1}b}$ can be shown to imply its own converse as follows.  If  $\overline{a^{-1}} \sim \overline{a^{-1}b}$ and $\bar{a}\neq \bar{b}$ then
\begin{equation}
\bar{a} = \overline{(a^{-1})^{-1}} \sim \overline{(a^{-1})^{-1}a^{-1}b} = \bar{b}.
\end{equation}
Thus we may replace property (1) with one involving a bidirectional implication.

As a consequence of this observation, we can characterize the equivalence class of $\overline{g^{-1}}$ in terms of that of $g^{-1}$.  Specifically this equivalence class is
\begin{equation}
\{\overline{g^{-1}}\} \cup \{\overline{g^{-1}h} \mid \bar{h}\sim \bar{g}, \bar{h}\neq\bar{g}\}.
\end{equation}
\end{rmk}

\subsubsection{Group actions on matroids of rank 3 with $G=\mathbb{Z}_p$}

To show the richness of simple rank $3$ $G$-invariant matroid structures with ground set $G$, we will briefly consider the simplest case of prime cyclic groups.  The following examples (such as Proposition \ref{proposition: cyclic}) hint that such matroids have a deep connection with number theory.

\begin{lem}\label{lemma: equivalence rel residue}
Let $p$ be a prime such that $p\equiv 7 \mathrm{(mod\ 12)}$.  Then there is a primitive 6th root of unity in $\mathbb{Z}_p$, which we call $u$. Define a relation $\sim$ on $\mathbb{Z}_p^\times$ as follows: for $x, y \in \mathbb{Z}_p^\times$
\[
x \sim y \iff \begin{cases}
    x=y, &\textrm{ if $x$ and $y$ are both quadratic residues or quadratic non-residues, } \\
    z=\pm w, & \textrm{if $x=z^2$ and $y=uw^2$, }\\
    z=\pm w & \textrm{if $x=uz^2$ and $y=w^2$. }
\end{cases}
\]
Then, $\sim$ is a nontrivial equivalence relation on $\mathbb{Z}_p^\times$. 
\end{lem}
\begin{proof}
Observe that $6 \mid (p - 1)$ so there is a primitive 6th root of unity, which we call $u$.  
Since $12 \nmid (p - 1)$, there is no primitive 12th root of unity, so $u$ cannot be a residue, i.e., there is no $x \in \mathbb{Z}_p$ such that $x^2 \equiv u$ (mod $p$).

Next, $\sim$ is clearly reflexive and symmetric. To show that $\sim$ is transitive, suppose that $x \sim y$ and $y\sim z$. If $x$ and $y$ are both residues or both non-residues, then there is nothing to prove as $x=y$. So, let's first assume that $x$ is a residue, $y$ is a non-residue, and $z$ is a residue. So, we have
 \[
 x=\alpha^2, \quad y=u\beta^2, \quad z=\gamma^2\quad  \textrm{ with } \alpha=\pm \beta = \pm \gamma,
 \]
showing that $x=z$, and hence $x \sim z$.

The remaining case is when $x,z$ are non-residues and $y$ is a residue. One can apply a similar argument as above to prove this case. 

Finally, $\sim$ is nontrivial since $p > 3$ implies that there are at least $2$ residues.
\end{proof}

\begin{lem}\label{lemma: sum of squares}
Let $p$ be a prime with $p\equiv 7 \mathrm{(mod\ 12)}$. Let $\sim$ be the equivalence relation in Lemma \ref{lemma: equivalence rel residue} with a fixed primitive 6th root of unity $u$. Then, for $x,y \in \mathbb{Z}_p^\times$, there exist $a,b \in \mathbb{Z}_p^\times$ such that $(x-y)=(a^2-b^2)$.
\end{lem}
\begin{proof}
 Setting $a = (vx - vy + v^{-1}) / 2$ and $b = (vx - vy - v^{-1}) / 2$ for any $v^{-1}\in \mathbb{Z}_p^\times$ will work unless $a$ or $b$ is zero.  In the latter case, $(x - y) = \pm v^{-2}$.  Choosing a different unit $w$ shows that if the claim is false then $\pm v^{-2} = (x - y) = \pm w^{-2}$, which implies there are at most two quadratic residues, implying $p \leq 5$.  But this contradicts $p\equiv 7 \mathrm{(mod\ 12)}$.  This proves the lemma. 
\end{proof}

\begin{pro}\label{proposition: cyclic}
Let $p$ be a prime with $p\equiv 7 \mathrm{(mod\ 12)}$. For each primitive 6th root of unity $u$, there is a simple $\mathbb{Z}_p$-invariant rank 3 matroid structure on $\mathbb{Z}_p$ such that a 3-element subset is a basis unless it is a translate (by addition) of $\{0, \alpha^2, u\alpha^2\}$ for some $\alpha\in \mathbb{Z}_p^\times$, i.e., $\{x,x+\alpha^2,x+u\alpha^2\}$.
\end{pro}
\begin{proof}

We prove that the equivalence relation $\sim$ in Lemma \ref{lemma: equivalence rel residue} satisfies the conditions in Proposition \ref{proposition: simple rank 3 classification} with $H=\{0\}$, and hence it defines a $\mathbb{Z}_p$-invariant matroid $M$ on $\mathbb{Z}_p$. As $H = 0$, the second condition of Proposition \ref{proposition: simple rank 3 classification} is vacuous.

Before we prove the first condition, we first prove that these sets are indeed bases. Recall from the proof of Proposition \ref{proposition: simple rank 3 classification} that $\{x,y,z\} \subseteq \mathbb{Z}_p^\times$ is a basis if and only if $(y-x) \not \sim (z-x)$. Now, since $|\{x,y,z\}|=3$, $(y-x)$ and $(z-x)$ can not be both residues or both non-residues. We may assume that $y-x= \alpha^2$ and $z-x= u\alpha^2$. It follows that $\{x,y,z\}=\{x,x+\alpha^2,x+u\alpha^2\}$, showing these sets are indeed bases of a matroid (provided it is actually a matroid). 


Now, we prove the first property of Proposition \ref{proposition: simple rank 3 classification}. 
Let $x,y\in\mathbb{Z}_p^\times$ be such that $x \sim y$ and $x\neq y$.  We need to show that 
\[
(-x) \sim (y - x).
\]
Since $x$ and $y$ are equivalent but not equal, one is a residue and the other is a non-residue.  We will make frequent use of the fact that $(-1)$ is a non-residue since $p\equiv 3 \mathrm{(mod\ 4)}$. Also, we will use the fact that $u^2-u+1=0$. 

We first consider the case that $x$ is the residue, so that $x\sim y$ implies that $x=z^2$, $y=uw^2$, and $z=\pm w$. Since $(u-1)=(-1)u^{-1}$ is the product of two non-residues, it is a residue. Hence, 
\[
(y-x)=uw^2-w^2=(u-1)w^2=(u-1)x
\]
is a residue. Moreover, by definition of $\sim$, we have $(y-x) \sim u(y-x)$. Therefore, we have
\[
(y-x) \sim u(y-x) = u(u-1)x = (-x).
\]
Next we consider the case that $y$ is the residue, so that $x\sim y$ implies that $x=uw^2$, $y = z^2$, and $w^2=z^2$.   
Then $(-x)$ is a residue.  Moreover, with $y=u^{-1}x$, we have
\[
(y - x) = (u^{-1} - 1)x = (-ux) \sim (-x), 
\]
showing that $(y - x) \sim (-x)$.
\end{proof}

The equivalence classes of the relation corresponding to a simple rank $3$ $\mathbb{Z}_p$-invariant matroid structure on $\mathbb{Z}_p$ have strong additive combinatorial properties.  Specifically, they are difference-free and differences between pairs of elements are distinct. To prove this, we use the following lemma.

\begin{lem}\label{lemma: first claim}
Let $G$ be a finite abelian group and $\sim$ be a nontrivial equivalence relation on $G - \{0\}$ such that for $x \neq y$, the equivalence $x\sim y$ implies the following equivalence:
\begin{equation}\label{eq: equi rel}
(y - x) \sim (-x).
\end{equation}
Suppose there is some $z$ with $z\sim (-z)$, and denote the equivalence class by $[z]$.  Then $[z] \cup \{0\}$ is a proper subgroup.
\end{lem}
\begin{proof}
Let $t$ be the order of $z$. First we show by induction that $z \sim (-kz)$ for $k=1,\ldots, t - 2$. The base case holds by assumption, so assume $z \sim (-kz)$, for some $k \geq 1$. Since the order of $z$ is $t$, we have that $z \neq -kz$ for $k \in \{1,\dots,t-2\}$. So by applying relation~\eqref{eq: equi rel}, we have $(-kz - z) \sim (-z)$, which implies the claim for $k + 1$.  In particular the cyclic group generated by $z$ is contained in $[z] \cup \{ 0 \}$.

Next, we prove that $[z]\cup \{0\}$ is closed under negation. Suppose that $x \sim z$. If $x \in \angles{z}$, then there is nothing to prove. So, we may assume that $x \not \in \angles{z}$. Then, since $x\sim z\sim (-z)$, we get
\begin{equation}\label{eq: nec1}
(x + z) \sim z \sim x,
\end{equation}
which yields
\begin{equation}\label{eq: nec2}
z = (x + z) - x\sim (-x),
\end{equation}
showing that $[z]\cup \{0\}$ is closed under negation. 

It remains to show that $[z] \cup \{0\}$ is closed under addition. Note that since any element of the equivalence class satisfies $x\sim (-x)$, all of the above holds with $x$ in place of $z$. Now, suppose that $x,y \in [z] \cup \{0\}$. We may assume that $x,y \neq 0$. If $x=y$, then, from the first paragraph, we have
\[
(x+y)=2x \in \angles{x} \subseteq [z]\cup\{0\}.
\]
If $x = -y$, then $x+y=0 \in [z]\cup\{0\}$. Finally, if $x \neq \pm y$, then we may apply the same argument as in the second paragraph (with $y$ instead of $z$) to obtain
\[
(x+y) \sim y \sim x, 
\]
showing that $(x+y)\sim z$, or $(x+y) \in [z]\cup\{0\}$.

Thus the equivalence class (together with $0$) is closed under addition by any element of the equivalence class.  Since it is closed under addition and negation, it is a subgroup. The subgroup $[z]\cup \{0\}$ is proper since otherwise the equivalence relation would be trivial.
\end{proof}

\begin{pro}\label{proposition: difference pro}
Let $p$ be an odd prime.  Let $\sim$ be a nontrivial equivalence relation on $\mathbb{Z}_p^\times$ such that for $x \neq y$, the equivalence $x\sim y$ implies the following equivalence:
\[
(y - x) \sim (-x).
\]
Then the following hold. 
\begin{enumerate}
    \item 
Let $x\in\mathbb{Z}_p^\times$.  Then $x\not\sim (-x)$.
    \item 
Let $x, y$ be distinct elements of some equivalence class.  Then $(y - x)$ does not belong to the equivalence class.
    \item 
Let $x,y\in\mathbb{Z}_p^\times$ be such that $x \sim y$ but $x\neq y$.  Then $(-x)\not\sim (-y)$.
    \item 
Let $x, y$ and $z, w$ be two pairs of distinct elements of some equivalence class.  Then if $x - y = z - w$, then $x = z$ and $y = w$.
\end{enumerate}
\end{pro}
\begin{proof}
The first claim follows by Lemma \ref{lemma: first claim}, since $x$ generates $\mathbb{Z}_p$.

Next, suppose $x,y$ satisfy $x\sim y$ and $x\neq y$.  We have
\[
(y - x) \sim (-x) \not\sim x.
\]
So $(y - x)\not\in [x]$, the equivalence class of $x$.  We also have
\[
(-y) \sim (x - y) \not\sim (y - x) \sim (-x).
\]
Note that $(x-y) \not \sim (y-x)$ from (1). This shows (2) and (3). 

For (4), suppose $x, y$ and $z, w$ are two pairs of distinct elements of some equivalence class.  We have
\[
(-x) \sim (y - x)=(w-z) \sim (-z). 
\]
From (3), one concludes that $x = z$ and $y=w$.  
\end{proof}

If we do not assume that $p$ is prime, then the proof of Proposition \ref{proposition: difference pro} shows that the only obstruction to the result holding is that there may be an element satisfying $x\sim (-x)$ (in which case none of the claims of the proposition needs to hold).

In order to produce an example where there is some $x$ with $x\sim (-x)$, we consider the following construction.

\begin{pro}
Let $X$ be a finite set and $\equiv$ be an equivalence relation and $k$ be a natural number. Consider
\[
\mathcal{B}=\{B \subseteq X \mid \textrm{$|B|=k$ and not all $k$ elements in $B$ are equivalent}\}. 
\]
Suppose that $\mathcal{B} \neq \emptyset$. Then $(X,\mathcal{B})$ is a matroid. 
\end{pro}
\begin{proof}
We only have to check the bases exchange axiom. Let $B_1=\{x_1,\dots , x_k\}$ and $B_2=\{y_1,\dots, y_k\}$ be bases. We may assume the element we are exchanging from $B_1$ is $x_k$. 

First, suppose that $x_i \sim x_j$ for all $i,j \in \{1,\dots, k-1\}$. Since $\{y_1, ..., y_k\}$ are not all equivalent, there is some element $y_i$ that is not equivalent to $x_1$. In particular, $y_i \neq x_j$ for $j \in \{1,\dots,k-1\}$ and $y_i \not \sim x_1$, and hence $\{x_1,\dots,x_{k-1},y_i\}$ is a basis.  

Now, suppose that $\{x_1, ..., x_{k - 1}\}$ are not all equivalent. Take any $y_i \in B_2 \backslash B_1$. Then $\{x_1, ..., x_{k - 1}, y_i\}$ is a basis.
\end{proof}

We specialize this to the case $X = \mathbb{Z}_n$, $\equiv$ is congruence modulo $d$ for some $d\mid n$.  Then the resulting matroid corresponds to the equivalence relation $\sim$ where $x\sim y$ means $x = y$ or $d | \mathrm{gcd}(x,y)$.  We see that $d\sim (-d)$. Here are some specific examples. 

\begin{myeg}
Let $X=\mathbb{Z}_4$ and $k=d=2$. Then we have
\[
1 \equiv 3 \textrm{ and }  0 \equiv 2. 
\]
So, we have 
$\mathcal{B}=\{\{1, 0\},\{1,2\},\{3, 0\}, \{3, 2\}\}$
So, $(X, \mathcal{B})=U_{1,2} \oplus U_{1,2}$. 

For $k=3$, we have
\[
\mathcal{B'}=\{\{1,2,3\}, \{0, 1, 3\}, \{0,1,2\}, \{0,1,3\}\} 
= U_{3,4}.
\]
\end{myeg}

The following example shows that our construction is different from partition matroids. 

\begin{myeg}
Let $X=\mathbb{Z}_{9}$, $k=3$, and $d=3$. Then, we have
\[
0 \equiv 3 \equiv 6, \quad 1 \equiv 4 \equiv 7, \quad 2 \equiv 5  \equiv 8.
\]
We claim that $(X,\equiv)$ is not a partition matroid. For the sake of contradiction, assume that $M:=(X,\equiv)$ is a partition matroid. We will use the term \textit{block} for a block of the corresponding partition. The \textit{capacity} of a block will be the size of the intersection of the block with a basis - it does not depend on the choice of basis by definition of partition matroid. 

Since $M$ has no loops, there is no block with capacity zero. Moreover, $M$ has no coloops as there is no element which belongs to every basis, and hence every block has at least two elements since every block has capacity at least one. 

Suppose there is a block with capacity 1, and let $x, y$ be distinct elements of this block.  Choose $z$ so that $z$ is not congruent to $x$ or $y$ mod 3.  Then $\{x, y, z\}$ is a basis, which contradicts the capacity assumption. So all blocks have capacity at least 2.  Since the capacities sum to the rank of the matroid $M$ (which is 3), this means there is a single block of size $3$.  That is the matroid is uniform.  But this also leads to a contradiction because $\{0, 3, 6\}$ is not a basis. 
\end{myeg}

\begin{rmk}
Proposition \ref{proposition: difference pro} can be used to derive bounds on the size of the equivalence classes.  To derive a simple statement along these lines, let $\{x_1,\ldots,x_k\}$ be an equivalence class.  Then the first and third parts of the proposition show that $(-x_1),\ldots,(-x_k)$ have distinct equivalence classes, none of which is the equivalence class we started with.  Moreover, all of these equivalence classes have size $k$ by Remark \ref{remark: equivalence class of inverse}.  Since there are at least $k + 1$ equivalence classes of size $k$, we have $k(k + 1) \leq p$.
\end{rmk}

We can subsume most of the above properties in a simpler statement by adjoining $0$ to the equivalence class.  Doing so uncovers a connection with the theory of Golomb rulers.

\begin{mydef}$S\subseteq \mathbb{Z}_p$ is said to be a \emph{modular Golomb ruler} if it contains $0$ and if for any two pairs of distinct residue classes $x\neq y$ and $z\neq w$ with $x - y = z - w$ we have $x = z$ and $y = w$.
\end{mydef}

\begin{cor}\label{cor: distinct difference property of equivalence class}
Let $p$ be an odd prime.  Let $\sim$ be a nontrivial equivalence relation on $\mathbb{Z}_p^\times$ such that $x \sim y$ but $x\neq y$ implies $(y - x) \sim (-x)$.  Let $S$ be the union of some equivalence class and $\{0\}$.  Let $x, y$ and $z, w$ be two pairs of distinct elements of $S$.  Then if $x - y = z - w$, we have $x = z$ and $y = w$.  In other words, $S$ is a modular Golomb ruler.
\end{cor}
\begin{proof}
If none of $x,y,z,w$ are zero, then this is just (4) of Proposition \ref{proposition: difference pro}. If $x=z=0$ or $y=w=0$, then this is vacuous. 

Next, when some of elements are zero, from symmetry, it is enough to consider the case that $x=0$. Then we have $y=w-z$. If $z=0$, then there is nothing to prove. Suppose that $z \neq 0$. If $w=0$, then we have $x=w=0$ and $y$ and $z$ are nonzero. But, then $x-y=-y=z=z-w$, in particular
\[
y \sim z \sim (-y),
\]
contradicting (1) of Proposition \ref{proposition: difference pro}. So, $w \neq 0$. But, then we have
\[
y = w-z \sim (-z),
\]
and hence we obtain the same contradiction. Therefore, if $x=0$, then $z=0$, and hence $x=z$ and $y=w$.
\end{proof}

As a generalization of the theory of Golomb rulers, we make the following definition.  Note that in the setting of the above corollary, the single set $S$ forms a distinct difference system.
\begin{mydef}\label{definition: difference system}
Let $A$ be an abelian group and $S_1, \ldots S_k\subseteq A$ be subsets.  We say $S_1,\ldots,S_k$ form a \emph{distinct difference system} in $A$ if the following hold.
\begin{itemize}
    \item 
    For any $i, j$ and any $x,y\in S_i$ and $w, z\in S_j$ with $x\neq y$ and $z\neq w$, the equation $x - y = z - w$ implies $x = z$ and $y = w$.
    \item 
    Each $S_i$ contains $0$ and at least one other element.
    \item $S_i \cap S_j = \{0\}$ for any $i\neq j$.
\end{itemize}
\end{mydef}


\begin{lem}\label{lemma: translates of distinct difference system}Let $A$ be an abelian group and $S_1,\ldots,S_k$ be a distinct difference system in $A$.  For any $i=1,\ldots, k$ and any $z\in S_i$ define
\begin{equation}
T_{i, z} = \{ x - z \mid x \in S_i, x\neq z \}.
\end{equation}
If $T_{i, z} \cap T_{j, w} \neq \emptyset$ then $i = j$ and $z = w$.
\end{lem}
\begin{proof}
Suppose that $a\in T_{i, z} \cap T_{j, w}$.  Then, there is some $x\in S_i$ and $y\in S_j$ such that
\begin{equation}
a = x - z = y - w
\end{equation}
which implies $x = y$ and $z = w$.  If $i = j$, the result holds.  Otherwise $x, z \in S_i \cap S_j = \{0\}$.  But this contradicts the definition of $T_{i, z}$, which requires $x\neq z$.
\end{proof}

\begin{pro}\label{proposition: distinct difference system matroid}Let $A$ be a finite abelian group with at least $3$ elements and $S_1,\ldots,S_k$ be a distinct difference system in $A$.  Define a basis to be a subset of $A$ consisting of $3$ elements which is not a translate of a subset of any of the $S_i$.  This yields a simple rank $3$ $A$-invariant matroid structure on $A$.
\end{pro}
\begin{proof}




Let $T_{i, z}$ be defined as in Lemma \ref{lemma: translates of distinct difference system}, and observe that these sets are all disjoint.

Define $x\sim y$ to mean that either $x = y$ or there is some $i$ and some $z\in S_i$ such that $x, y\in T_{i, z}$.  The only non-obvious step in showing that $\sim$ is an equivalence relation is to establish transitivity in the case where the 3 elements are distinct.  Let $a, b, c$ be distinct with $a\sim b\sim c$.  Then there is some $i, j$, and some $z\in S_i$, and $w\in S_j$ such that $a, b\in T_{i, z}$ and $b, c\in T_{j, w}$.  Since $b\in T_{i, z}\cap T_{j, w}$ we get $i = j$ and $z = w$ so $a, c\in T_{i, z}$.  This implies $a\sim c$.

Now, we show that the above equivalence relation satisfies the two conditions in Proposition \ref{proposition: simple rank 3 classification} with $H=\{0\}$ and is nontrivial. Since $H=\{0\}$, the second condition is vacuous. So, we only have to show that the equivalence relation satisfies the first condition and is nontrivial.

Suppose that $x\sim y$ but $x\neq y$.  There is some $z$ such that $x, y\in T_{i, z}$.  In particular $z, x + z, y + z\in S_i$.  This implies that $T_{i, x + z}$ contains $(-x) = z - (x + z)$ and $(y - x) = (y + z) - (x + z)$.  Thus $(-x) \sim (y - x)$.

 Next, we claim that $\sim$ is nontrivial.  Suppose otherwise and let $x\neq 0$.  Then for any $y\not\in \{0, x\}$ we have $x\sim y$ so there is some $T_{i, z}$ with $x, y\in T_{i, z}$.  Since there is only one $T_{i, z}$ containing $x$, it contains all nonzero elements. It follows that $S_i = A$.  But this contradicts the distinct difference property since $2x - x = x - 0$ but $x \neq 0$.

Now $\sim$ corresponds to a simple rank $3$ $A$-invariant matroid structure on $A$.  Consider an arbitrary set $\{w, x + w, y + w\}$ of three distinct elements.  This is dependent if and only if $x\sim y$.  Since $x\neq y$, this is equivalent to the existence of $i$ and $z\in S_i$ such that $x, y\in T_{i, z}$.  Since $x, y\neq 0$, this is equivalent to the existence of $i$ and $z\in S_i$ such that $x + z, y + z \in S_i$.  This in turn is equivalent to the existence of $z\in A$ and $i$ such that $z, x + z, y + z \in S_i$.  Since $\{z, x + z, y + z\}$ is an arbitrary translate of $\{w, x + w, y + w\}$, this is equivalent to $\{w, x + w, y + w\}$ being a translate of a $3$-element subset of some $S_i$. This shows our description of bases is equivalence to the one obtained from the equivalence relation via Proposition~\ref{proposition: simple rank 3 classification}.
\end{proof}

\begin{myeg}\label{example: matroid cyclic}
Let $n$ be an odd natural number.  Let $k,u\in \mathbb{Z}_n^\times$. 
Define 
\[
S = \{0, u, ku \}.
\]
Then the single subset $S\subseteq \mathbb{Z}_n$ forms a distinct difference system as long as
\[
0, \quad u, \quad ku, \quad -u, \quad -ku, \quad (k - 1)u,\quad  (1 - k)u
\]
are all distinct.  Since $u$ is a unit, this requirement is equivalent to $0, 1, k, (-k), (k - 1), (1 - k)$ being all distinct.  This holds if $k\neq 0, \pm 1, 2, \frac{n + 1}{2}$.  
In this case we obtain a matroid structure in which a 3 element subset is a basis unless it is a translate of $\{0, u, ku\}$.  This produces the matroid of \cite[Theorem 4.9 (1)]{marcus2024tropical}.  The case of even $n$ is similar, which is \cite[Theorem 4.9 (2)]{marcus2024tropical}.
\end{myeg}

By working directly with matroids rather than equivalence classes, we can extend our construction to ranks greater than $3$.  

\begin{mythm}\label{theorem: abelian case translation}
Let $k\geq 3$.  Let $A$ be a finite abelian group with at least $k$ elements.  Fix a distinct difference system $S_1,\ldots S_l$.  Define a subset of $A$ to be a basis if it has $k$ elements and is not a translate of a subset of some $S_i$.  Then this yields a simple rank $k$ $A$-invariant matroid structure on $A$.
\end{mythm}
\begin{proof}The $A$-invariance is trivial.  To show simplicity, we will prove the stronger claim that every subset of size less than $k$ is independent.  For this, it suffices to consider a subset $I\subseteq A$ of $k - 1$ elements.  For the sake of contradiction, we suppose that $I \cup \{u\}$ is not a basis for any $u\in A \backslash I$.  We observe that $S_i \neq A$ for any $i$, because $S_i$ cannot contain a subset of the form $\{0, s, s + s\}$ by 
the first condition in Definition \ref{definition: difference system} applied to the case $x=0$, $y=s$, $z=s$, and $w=s+s$ with $i=j$.

We first consider the case when $A$ has exactly $k$ elements, so there is only one choice of $u$ and $I \cup \{u\} = A$.  Since $I \cup \{u\}$ is dependent, it is a translate of a subset of some $S_i$, which implies that $S_i = A$ for some $i$, giving us a contradiction.

Now suppose that $A$ has at least $(k+1)$ elements. Fix $u\not\in I$.  Since $I \cup \{u\}$ is not a basis, there is some $i$ and $a\in A$ such that
\begin{equation}
I\subseteq I \cup \{u\} \subseteq \{ s + a \mid s\in S_i \}.
\end{equation}
For any $v\not\in I$ with $v\neq u$, there is some $j$ and $b\in A$ such that
\begin{equation}
I\subseteq I \cup \{v\} \subseteq \{ t + b \mid t\in S_j \}.
\end{equation}
Let $x, y\in I$ be distinct (which is possible since $k - 1\geq 2$).  Then $(x - a), (y - a)\in S_i$ and $(x - b), (y - b)\in S_j$.  Observe that
\begin{equation}
(x - a) - (y - a) = (x - b) - (y - b).
\end{equation}
Since $(x - a)\neq (y - a)$, by the definition of distinct difference system, we must have $(x - a) = (x - b)$ and $(y - a) = (y - b)$.  So $a = b$.  But then $S_i$ and $S_j$ both contain $(x - a)$ and $(y - a)$, so we must also have $i = j$.  Now
\begin{equation}
v\in I \cup \{v\}\subseteq \{ t + b \mid t\in S_j \} = \{ s + a \mid s\in S_i \}
\end{equation}
Since this holds for all $v\not\in I \cup \{u\}$ and since $I \cup \{u\}$ is also contained in the right side of the above equation, the right side is $A$.  Thus, we have
\[
A=\{s+a \mid s \in S_i\}.
\]
It follows that $S_i$ has a subset of the form $\{0,t,t+t\}$,  which is a contradiction.  Thus every set with less than $k$ elements is contained in a basis; in particular such a set cannot be a circuit.

Finally, we must prove this yields a matroid, which we do via the circuit elimination axiom.  Let $C_1, C_2\subseteq A$ be circuits containing a common element $e\in C_1\cap C_2$.  By the above, each contains at least $k$ elements.  We must show that $(C_1 \cup C_2) \backslash \{e\}$ is dependent or $C_1 = C_2$.  Note that this is trivial if either circuit has more than $k$ elements, since then $(C_1 \cup C_2) \backslash \{e\}$ also has more than $k$ elements (assuming $C_1\neq C_2$), and hence cannot be contained in a basis.

So we suppose both circuits have $k$ elements.  Note that
\begin{equation}
| (C_1 \cup C_2) \backslash \{e\} | = |C_1| + |C_2| - |C_1\cap C_2| - 1 = 2k - 1 - |C_1 \cap C_2|.
\end{equation}
If $(C_1 \cup C_2) \backslash \{e\} $ has more than $k$ elements, then it is dependent, while if it has fewer than $k$ elements then $|C_1 \cap C_2| > (k - 1)$ so $C_1 = C_2$.  So we may assume it has $k$ elements. In particular, we may assume that $|C_1 \cap C_2|=k-1 \geq 2$.

Let $x,y\in C_1\cap C_2$ be distinct.  Since $C_1, C_2$ have size $k$ and are not bases, there exists $i, j$ and $a,b\in A$ such that
\begin{equation}
C_1 \subseteq \{ s + a \mid s\in S_i \}
\end{equation}
and
\begin{equation}
C_2 \subseteq \{ t + b \mid t\in S_j \}.
\end{equation}
In particular $(x - a), (y - a)\in S_i$ and $(x - b), (y - b) \in S_j$.  As before, we may use
\begin{equation}
(x - a) - (y - a) = (x - b) - (y - b)
\end{equation}
to obtain $a = b$.  Then $(x - a), (y - a)$ are both in $S_i \cap S_j$ so $i = j$.  Now, we have
\begin{equation}
(C_1 \cup C_2) \backslash \{e\} \subseteq C_1 \cup C_2 \subseteq \{ s + a \mid s\in S_i \}.
\end{equation}
Since $(C_1 \cup C_2) \backslash \{e\}$ has $k$ elements, and since the above equation says it is not a basis, it must be dependent, which establishes the theorem.
\end{proof}

Returning to the case of simple rank $3$ $A$-invariant matroid structures on $A$, we will next show that if $A=\mathbb{Z}_p$, then all such matroid structures arise via our construction.

\begin{lem}\label{lemma: distinct difference pair dichotomy}
Let $p$ be an odd prime.  Let $\sim$ be an equivalence relation on $\mathbb{Z}_p^\times$ such that $x\sim y$ and $x\neq y$ together imply $(-x) \sim (y - x)$.  Let $S_1$ and $S_2$ each be obtained by adjoining $0$ to some equivalence class.  Then either $S_1, S_2$ form a distinct difference system or there is some $a\in S_1$ and $b\in S_2$ such that 
\begin{equation}
\{x - a\mid x\in S_1\} = \{y - b\mid y\in S_2\}.
\end{equation}
\end{lem}
\begin{proof}
Suppose $S_1, S_2$ is not a distinct difference system.  Then there exists some $i, j\in \{1,2\}$ and some $x, y\in S_i$ and $z,w\in S_j$ such that $x\neq y$, $z\neq w$, $x\neq z$, $y\neq w$ and $x - y = z - w$.  By Corollary \ref{cor: distinct difference property of equivalence class}, we must have $i\neq j$, and without loss of generality, we may assume $i = 1$ and $j=2$.

From Remark \ref{remark: equivalence class of inverse}, if $y\neq 0$ then
\begin{equation}
[-y] \cup \{ 0\} = \{ s - y \mid s \in S_1 \}    
\end{equation}
and if $w\neq 0$ then
\begin{equation}
[-w] \cup \{ 0\} = \{ t - w \mid t \in S_2 \}.
\end{equation}
Also observe that if $y\neq 0$, then $(x-y)\sim (-y)$ (this is reflexivity if $x=0$ and follows from $x\sim y$ and $x\neq y$ otherwise) and if $w\neq 0$ then $(z - w)\sim (-w)$.  In particular, if $y, w\neq 0$ then
\begin{equation}
(-y) \sim (x - y) = (z - w) \sim (-w).
\end{equation}
Hence, we have
\begin{equation}
    \{ s - y \mid s \in S_1 \}  = \{ t - w \mid t \in S_2 \},
\end{equation}
which proves the theorem in this case.

If $w = 0$ and $y\neq 0$, then 
\begin{equation}
(-y) \sim (x - y) = z
\end{equation}
and adjoining $0$ to the equivalence class of both sides yields
\begin{equation}
S_2 = [z] \cup \{0 \} = [-y] \cup \{ 0\} = \{ a - y \mid a \in S_1 \}
\end{equation}
The case $y = 0$ and $w\neq 0$ is handled by the same argument.

\end{proof}

\begin{mythm}\label{theorem: Classification in the prime cyclic case}
Let $p$ be an odd prime.  Every simple rank $3$ $\mathbb{Z}_p$-invariant matroid structure on $\mathbb{Z}_p$ can be produced by the construction from Proposition \ref{proposition: distinct difference system matroid}.
\end{mythm}
\begin{proof}
Fix a simple rank $3$ $\mathbb{Z}_p$-invariant matroid structure on $\mathbb{Z}_p$ and let $\sim$ be the corresponding equivalence relation. Among all distinct difference systems $S_1$,\ldots, $S_k$ with the additional property that each $S_i - 0$ is an equivalence class, fix one that it maximal with respect to inclusion.\footnote{Corollary \ref{cor: distinct difference property of equivalence class} ensures that there exists at least one such distinct difference system, and since everything is finite, such a maximal distinct difference system exists.} 

For each $i$ and each $z\in S_i$, let $T_{i,z}$ be as in Lemma \ref{lemma: translates of distinct difference system}.  Note that these subsets are disjoint.  We claim that they form a partition of $\mathbb{Z}_p^\times$.  Let $x\in \mathbb{Z}_p^\times$ and $S = [x] \cup \{ 0 \}$.  We must show that there is some $i, z$ such that $x\in T_{i, z}$.

By maximality $S_1,\ldots,S_k,S$ is not a distinct difference system.  Since each axiom of a distinct difference system involves only two subsets at a time, there is some $i$ such that $S_i, S$ is not a distinct difference system.  By Lemma \ref{lemma: distinct difference pair dichotomy}, there is some $a\in S$ and $b\in S_i$ such that
\begin{equation}
\{ s - a \mid s\in S \} = \{ t - b \mid t\in S_i \}.
\end{equation}

We shall show $b - a\in S_i$.  If $a = 0$ or if $a = b$, this is clear.  If $b = 0$, then $b-a = -a$ belongs to the left side of the equation and the right is $S_i$.  Otherwise, each side contains $-a$ and $-b$ (since $0\in S\cap S_i$).  By Remark \ref{remark: equivalence class of inverse}, each side is the union of $0$ and an equivalence class.  Thus $(-a) \sim (-b)$, which implies $b\sim (b - a)$.  Since $S_i$ is the union of 0 and an equivalence class and $b\in S_i$, we get $b - a \in S_i$.

We add $a$ to both sides of the above equation to obtain
\begin{equation}\label{eq: subt}
S = \{ t - (b - a) \mid t \in S_i \}.
\end{equation}
Since $b-a\in S_i$, removing $0$ from both sets in \eqref{eq: subt} yields $[x] = T_{i, b - a}$, establishing the claim.

Next we claim the $T_{i, z}$ are precisely the equivalence classes of $\sim$.  Since they cover $\mathbb{Z}_p$, it suffices to show each $T_{i, z}$ is an equivalence class.  If $z = 0$, $T_{i, 0} = S_i - 0$ is an equivalence class by assumption.  Otherwise $T_{i, z} $ is the equivalence class of $-z$ by Remark \ref{remark: equivalence class of inverse}.

Now let $\equiv$ be the equivalence relation appearing in the proof of Proposition
\ref{proposition: distinct difference system matroid}.  Then $x\equiv y$ means $x = y$ or there is some $i$ and some $z\in S_i$ such that $x,y\in T_{i,z}$.  But we have seen there is some $i$ and $z$ such that $x\in T_{i,z}$.  So $x\equiv y$ is equivalent to the existence of $i, z$ such that $x,y\in T_{i,z}$.  That is the sets $T_{i, z}$ are the equivalence classes of $\equiv$.  Thus $\equiv$ agrees with $\sim$, which implies the construction from Proposition
\ref{proposition: distinct difference system matroid} yields the matroid we started with.


\end{proof}

\subsubsection{Classification - non-simple rank 3 case}

Next we relate tropical subrepresentations of $\mathbb{B}[G]$ to actions on simple matroids whose underlying $G$-set is cyclic.  The key to this is the same equivalence relation we used in the rank $2$ case.  We will show the quotient is a simple matroid.

\begin{lem}\label{lemma: simple matroid quotient}
Let $M$ be a loopless matroid on $E$ whose rank is at least $2$.  Write $x\sim y$ if $x=y$ or $\{x, y\}$ is dependent.  Then
\begin{enumerate}
    \item 
If $x \sim y$ and $I\subseteq E$ does not contain $x$ and $I + x$ is independent then $I + y$ is independent.
    \item 
If $I = \{x_1,\ldots, x_n\}$ and $I' = \{y_1,\ldots, y_n\}$ are such that $x_i \sim y_i$ for all $i$, then $I$ is independent if and only if $I'$ is independent.
    \item 
Let $\bar{E} = E/ \sim$, the set of equivalence classes. Consider the following set:
\[
\mathcal{I}:=\{ \bar{I} \subseteq \bar{E} \mid \textrm{ for some independent set $I \subseteq E$ of $M$}\},
\]
where $\bar{I}$ is the image of $I$ in $\bar{E}$. Then $(\bar{E},\mathcal{I})$ defines a simple matroid $\bar{M}$ on $\bar{E}$ with the independent sets $\mathcal{I}$. Furthermore, $\text{rank}(\bar{M})=\text{rank}(M)$.
\end{enumerate}
\end{lem}
\begin{proof}
Note that from Lemma \ref{lemma: equivalence relation and rank 2}, we know that $\sim$ is an equivalence relation and moreover since the rank of $M$ is assumed to be at least $2$, the equivalence relation $\sim$ is nontrivial from the transitivity of $\sim$.  

(1) and (2): For the sake of contradiction, suppose that $I + y$ is dependent, meaning it contains some circuit $C$.  Since $I+x$ is independent, we have $x \neq y$. Moreover, since $I$ is independent, as it is a subset of an independent set $I+x$, we have $y\in C$.  Because $x\sim y$, $\{x, y\}$ is also a circuit.  By circuit elimination, $C + x - y$ is dependent. But, as $C - y \subseteq I$, $I + x$ is dependent, giving us a contradiction. The second claim follows by applying the first claim repeatedly in order to replace each $x_i$ by $y_i$. Note that if $y_i=x_j$ for some $j \neq i$, then $x_i\sim x_j$ so $\{x_i, x_j\}$ is dependent, giving us a contradiction.


(3): Let $\bar{I}, \bar{J}$ be independent sets of $\bar{M}$.  It is clear any subset of $\bar{I}$ is also independent, so we must prove the exchange axiom.  Suppose $|\bar{I}| < |\bar{J}|$.  Fix a section $s: \bar{E}\rightarrow E$ of the quotient map.  $\bar{I} = \{\bar{x_1},\ldots,\bar{x_n}\}$ is the image of some independent set $\{x_1,\ldots,x_n\}\subseteq E$.  Since $x_i \sim s(\bar{x_i})$, we know that $s(\bar{I}) = \{ s(\bar{x_1}),\ldots,s(\bar{x_n}) \}$ is also independent by (2), and its image is $I$.

Since $s(\bar{J})$ contains $|\bar{J}| > |\bar{I}| = |s(\bar{I})|$ elements, there exists some $s(u)\in s(\bar{J}) - s(\bar{I}) = s(\bar{J} - \bar{I})$ such that $s(\bar{I} + u) = s(\bar{I}) + s(u)$ is independent.  Since its image under the quotient map is $\bar{I} + u$, we obtain that $\bar{I} + u$ is independent, establishing the exchange axiom.

It is clear that the rank of $\bar{M}$ is the rank of $M$ since no two elements of a basis are equivalent.

Let $\{\bar{x}, \bar{y}\}\subseteq\bar{M}$ be a two element subset.  Since $\bar{x} \neq \bar{y}$, $x\not\sim y$ so $\{x, y\}$ is independent.  Thus $\{\bar{x}, \bar{y}\}$ is independent, so $\bar{M}$ is simple.
\end{proof}

We now show that loopless matroids are essentially sets with a quotient that carries the structure of a simple matroid.

\begin{pro}\label{prop: correspondence between loopless and simple matroids}
Let $E$ be a finite set.  Let $k\geq 2$.  
\begin{enumerate}
    \item 
There is a one-to-one correspondence between loopless rank $k$ matroid structures on $E$ and pairs consisting of a nontrivial equivalence relation $\sim$ on $E$ and a simple rank $k$ matroid structures on $E/\sim$.
    \item 
Let $G$ be a group and fix a $G$-action on $E$.  There is a one-to-one correspondence between $G$-invariant loopless rank $k$ matroid structures on $E$ and pairs consisting of a nontrivial $G$-invariant equivalence relation $\sim$ on $E$ and a simple rank $k$ $G$-invariant matroid structures on $E/\sim$.
\end{enumerate}
\end{pro}
\begin{proof}
(1): Given a loopless rank $k$ matroid structure on $E$, we obtain a nontrivial equivalence relation $\sim$ and a simple rank $k$ matroid structure on $M/\sim$ by (3) of Lemma \ref{lemma: simple matroid quotient}.  Let $\pi: E \rightarrow E/\sim$ be the quotient map.  Observe that if $I\subseteq E$ is an independent subset of $M$, then $\pi(I)$ is independent subset of $\bar{M}$ by definition, where $\bar{M}$ is as in Lemma \ref{lemma: simple matroid quotient}. Also by definition of $\sim$, an independent subset $I$ contains at most one element of each fiber of $\pi$.  

Next, we check the injectivity of the correspondence. For this, it is enough to show that $M$ is uniquely determined by the above equivalence relation $\sim$ and the matroid $\bar{M}$ on $E/\sim$, considered above. Suppose $I\subseteq E$ contains at most one element of each fiber and is such that $\pi(I)$ is independent.  Then we may choose an independent subset $I'\subseteq E$ such that $\pi(I') = \pi(I)$.  And as above, $I'$ contains at most one element of each fiber.  So by (2) of Lemma \ref{lemma: simple matroid quotient}, the independence of $I'$ implies that of $I$.  
Thus $I\subseteq E$ is an independent subset of $M$ if and only if $\pi(I)\subseteq E/ \sim$ is an independent subset of $\bar{M}$.  This implies the matroid structure of $M$ is determined by $\sim$ and the matroid $\bar{M}$.  That is the correspondence is injective.

Now suppose we are given a nontrivial equivalence relation $\sim$ and a simple rank $k$ matroid structure $N$ on $E/\sim$.  We define $\mathcal{I}$ be the set of subsets $I \subseteq E$ such that $\pi(I)$ is an independent subset of $N$ and $I$ contains at most one element of each fiber of $\pi$. We claim that $(E,\mathcal{I})$ is a matroid with the independent subsets $\mathcal{I}$. In fact, clearly a subset of an independent set is independent.

For the exchange axiom, let $I, J \in \mathcal{I}$ such that $|I| < |J|$.  Then $|\pi(I)| < |\pi(J)|$ since each set contains only one element of each fiber.  Thus there exists some element $\bar{x}\in \pi(J) - \pi(I)$ such that $\pi(I) + \bar{x}$ is independent.  Pick some $x\in J$ such that $\pi(x)=\bar{x}$.  Clearly $x\not\in I$ and $I + x$ contains at most one element of each fiber.  Moreover we observed $\pi(I + x)$ is independent so $I + x$ is independent.  Thus $(E,\mathcal{I})$ is a matroid, and clearly the rank is $k$. Let's denote this matroid by $M$. Next observe that if $x\in E$ is a loop of $M$ then so is $\pi(x)\in E/\sim$, which is impossible.  Thus $M$ is loopless.  This shows that the correspondence is surjective. 

(2): We only have to check that the correspondence in $(1)$ is $G$-equivariant. Let $M$ be a $G$-invariant loopless rank $k$ matroid on $E$. Then, the induced equivalence relation $\sim$ is $x\sim y$ if $x=y$ or $\{x,y\}$ is dependent. Now, suppose that $x \sim y$. If $x=y$, then $gx=gy$, and hence $gx \sim gy$. If $\{x,y\}$ is dependent, then $\{gx,gy\}$ is dependent as well since otherwise $\{gx,gy\}\subseteq B$ for some basis $B$, and hence $\{x,y\}\subseteq g^{-1}B$, which is also a basis. The $G$-invariance of the matroid $\bar{M}$ on $E/\sim$ is clear from the definition of $\bar{M}$. 

Conversely, suppose that we have a nontrivial $G$-invariant equivalence relation $\sim$ on $E$ and a simple rank $k$ $G$-invariant matroid $\bar{M}$ on $E/\sim$. From (1), we know that $\bar{M}$ comes from a matroid $M$ on $E$ and $I\subseteq E$ is an independent subset of $M$ if and only if $\pi(I) \subseteq E/\sim$ is an independent subset of $\bar{M}$, where $\pi:E \to E/\sim$ is the projection map and $I$ contains at most one element of each fiber. 

Let $I=\{x_1,\dots,x_r\}$. First, we claim that $gx_i\neq gx_j$ for all $g \in G$ and $i \neq j$. Indeed, if $gx_i=gx_j$, then $gx_i \sim gx_j$, and since $\sim$ is $G$-invariant, we have $x_i \sim x_j$ which means that $\{x_i,x_j\}$ is dependent, giving us a contradiction. So, we have $gI=\{gx_1,\dots,gx_r\}$, and from the construction of $\bar{M}$ together with the assumption that $\bar{M}$ is $G$-invariant, $\pi(gI)$ is an independent subset of $\bar{M}$. It follows that $gI$ is an independent subset of $M$ as desired. 
\end{proof}

\begin{cor}\label{corollary: thmc}
Let $G$ be a finite group.  There is a one-to-one correspondence between rank $3$ $G$-invariant matroid structures on $G$ and pairs consisting of a proper subgroup $H\subseteq G$ and a nontrivial equivalence relation $\sim$ on $G/H - \{\bar{1}\}$ which satisfies the properties listed in Proposition \ref{proposition: simple rank 3 classification}. Furthermore, $H$ is the trivial group if and only if the matroid is simple.
\end{cor}
\begin{proof}
First we observe that every rank $3$ $G$-invariant matroid with ground $G$-set equal to $G$ has no loops.  This is because if $g$ is a loop, then by transitivity of the $G$-action on $G$, every element is a loop.  This contradicts the assumption that the matroid has rank $3$. 

Next observe that $G$-invariant nontrivial equivalence relations on $G$ are in one-to-one correspondence with proper subgroups of $G$.  So by Proposition \ref{prop: correspondence between loopless and simple matroids} rank $3$ $G$-invariant matroid structures on $G$ correspond to pairs consisting of a proper subgroup $H$ and a simple rank $3$ $G$-invariant matroid structure on $G / H$.  The result now follows from Proposition \ref{proposition: simple rank 3 classification}.
\end{proof}

\bibliography{Vectorbundle}\bibliographystyle{alpha}

\begin{thebibliography}{CGM20}

\bibitem[BB19]{baker2019matroids}
Matthew Baker and Nathan Bowler.
\newblock Matroids over partial hyperstructures.
\newblock {\em Advances in Mathematics}, 343:821--863, 2019.

\bibitem[BN24]{beardsley2024projective}
Jonathan Beardsley and So~Nakamura.
\newblock Projective geometries and simple pointed matroids as
  $\mathbb{F}_1$-modules.
\newblock {\em arXiv preprint arXiv:2404.04730}, 2024.

\bibitem[CGM20]{crowley2020module}
Colin Crowley, Noah Giansiracusa, and Joshua Mundinger.
\newblock A module-theoretic approach to matroids.
\newblock {\em Journal of Pure and Applied Algebra}, 224(2):894--916, 2020.

\bibitem[Fre13]{frenk2013tropical}
Bart Frenk.
\newblock Tropical varieties, maps and gossip.
\newblock {\em Ph.D. thesis, Technische Universiteit Eindhoven}, 2013.

\bibitem[GG18]{giansiracusa2018grassmann}
Jeffrey Giansiracusa and Noah Giansiracusa.
\newblock A {G}rassmann algebra for matroids.
\newblock {\em manuscripta mathematica}, 156(1-2):187--213, 2018.

\bibitem[GM20]{giansiracusa2020matroidal}
Noah Giansiracusa and Jacob Manaker.
\newblock Matroidal representations of groups.
\newblock {\em Advances in Mathematics}, 366:107089, 2020.

\bibitem[MP24]{marcus2024tropical}
Steffen Marcus and Cameron Phillips.
\newblock Tropical subrepresentations of the {B}oolean regular representation
  in low dimension.
\newblock {\em arXiv preprint arXiv:2410.08349}, 2024.

\bibitem[NR23]{nakamura2023categories}
So~Nakamura and Manuel~L Reyes.
\newblock Categories of hypermagmas, hypergroups, and related hyperstructures.
\newblock {\em arXiv preprint arXiv:2304.09273}, 2023.

\end{thebibliography}

\end{document}